\documentclass[12pt]{article}

\usepackage[a4paper, top=3cm, left=2.5cm, right=2.5cm, bottom=3cm]{geometry}
\usepackage{amsmath, amsthm}
\usepackage{amssymb}
\usepackage{multicol,comment}
\usepackage{amsfonts}
\usepackage{graphics, color, enumerate, fancyhdr, dsfont}

    \newcommand{\R}{\mathbb{R}}

    \newcommand{\Bwf}{\mathcal{B}}

    \newcommand{\Hwf}{\mathcal{H}}
    \newcommand{\Iwf}{\mathcal{I}}

    \newcommand{\Mwf}{\mathcal{M}}
    \newcommand{\Nwf}{\mathcal{N}}
    
    \newcommand{\Pwf}{\mathcal{P}}

    \newcommand{\Swf}{\mathcal{S}}

    \newcommand{\afrak}{\mathfrak{a}}
    \newcommand{\bfrak}{\mathfrak{b}}
    \newcommand{\cfrak}{\mathfrak{c}}
    \newcommand{\dfrak}{\mathfrak{d}}

    \newcommand{\gfrak}{\mathfrak{g}}

    \newcommand{\pfrak}{\mathfrak{p}}
    \newcommand{\rfrak}{\mathfrak{r}}
    \newcommand{\sfrak}{\mathfrak{s}}
    
    \newcommand{\ufrak}{\mathfrak{u}}

    \newcommand{\menos}{\smallsetminus}

    \newcommand{\frestr}{\!\!\upharpoonright\!\!}

    \newcommand{\add}{\mbox{\rm add}}
    \newcommand{\cov}{\mbox{\rm cov}}
    \newcommand{\non}{\mbox{\rm non}}
    \newcommand{\cof}{\mbox{\rm cof}}

    \newcommand{\Bor}{\mathds{B}}
    \newcommand{\Cor}{\mathds{C}}
    \newcommand{\Dor}{\mathds{D}}
    \newcommand{\Eor}{\mathds{E}}
    
    \newcommand{\Mor}{\mathds{M}}
    \newcommand{\Por}{\mathds{P}}
    \newcommand{\Qor}{\mathds{Q}}
    
    \newcommand{\Sor}{\mathds{S}}
    
    \newcommand{\Qnm}{\dot{\mathds{Q}}}
    \newcommand{\Rnm}{\dot{\mathds{R}}}

    \newcommand{\cf}{\mbox{\rm cf}}

\title{Some models produced by 3D iterations}
\author{Diego Alejandro Mej\'{\i}a}
\date{\small Faculty of Science\\ Shizuoka University\\ 836 Ohya, Suruga-ku, 422-8529 Shizuoka, Japan\\ \texttt{diego.mejia@shizuoka.ac.jp}}

\begin{document}

\makeatletter
\def\@roman#1{\romannumeral #1}
\makeatother

\theoremstyle{plain}
  \newtheorem{theorem}{Theorem}[section]
  \newtheorem{corollary}[theorem]{Corollary}
  \newtheorem{lemma}[theorem]{Lemma}
  \newtheorem{prop}[theorem]{Proposition}
  \newtheorem{claim}[theorem]{Claim}
  \newtheorem{exer}[theorem]{Exercise}
\theoremstyle{definition}
  \newtheorem{definition}[theorem]{Definition}
  \newtheorem{example}[theorem]{Example}
  \newtheorem{remark}[theorem]{Remark}
  \newtheorem{context}[theorem]{Context}
  \newtheorem{question}[theorem]{Question}
  \newtheorem{problem}[theorem]{Problem}
  \newtheorem{notation}[theorem]{Notation}

\maketitle

\newcommand{\la}{\langle}
\newcommand{\ra}{\rangle}
\newcommand{\id}{\mathrm{id}}
\newcommand{\sig}{\boldsymbol{\Sigma}}
\newcommand{\cosig}{\boldsymbol{\Pi}}

\newcommand{\leqdi}{\preceq_{\mathrm{di}}}
\newcommand{\eqdi}{\approx_{\mathrm{di}}}
\newcommand{\leqcdi}{\preceq_{\mathrm{cdi}}}
\newcommand{\eqcdi}{\approx_{\mathrm{cdi}}}
\newcommand{\eqPc}{\approx_{\mathrm{P}}}

\begin{abstract}
   We use the techniques in \cite{VFJB11,mejia2,FFMM} to construct models, by three-dimensional arrays of ccc posets, where many classical cardinal characteristics of the continuum are pairwise different.
\end{abstract}

\section{Introduction}\label{SecIntro}

For quite some time, researchers in set theory have been working on producing models of ZFC where more than two cardinal characteristics of the continuum are pairwise different. The first techniques in this direction are the preservation properties for cardinal characteristics in the context of FS (finite support) iterations by Judah and Shelah \cite{jushe}, later refined by Brendle \cite{Br-Cichon} to produce models for all the possible consistent constellations of Cicho\'n's diagram in two pieces (excluding $\aleph_1$) under the conditions $\aleph_1<\add(\Nwf)$ and $\non(\Mwf)\leq\cov(\Mwf)$ (this last restriction is unavoidable in the context of FS iterations). Other earlier example is produced by Blass and Shelah \cite{blassmatrix} where they force $\aleph_1<\ufrak<\dfrak$ by a FS iteration constructed by a two-dimensional array of ccc posets, technique usually referred to as \emph{matrix iterations}.

The first time the term \emph{matrix iterations} appeared was in \cite{VFJB11} where Brendle and V. Fischer used that technique to prove the consistency of $\bfrak=\afrak<\sfrak$ by assigning arbitrary regular values, and likewise the consistency of $\bfrak=\sfrak<\afrak$ but using a measurable cardinal in the ground model. Later on, the author \cite{mejia} induced Judah-Shelah-Brendle preservation theory into matrix iterations to produce models where several cardinal invariants in Cicho\'n's diagram are pairwise different. For instance, a model where those cardinal invariants are separated into 6 different values (including $\aleph_1$). Quite simultaneously, the author collected related techniques in \cite{mejia2} to produce models where other classical cardinal characteristics of the continuum are separated into several values, like $\pfrak$, $\sfrak$, $\rfrak$ and $\ufrak$.

There are other techniques outside the context of FS iterations of ccc posets to produce models where several cardinal invariants are pairwise different. A very remarkable one is \emph{large product constructions by creature forcing}. Goldstern and Shelah \cite{GSmany} produced a model where $\aleph_1$-many definable cardinal invariants are pairwise different, while Kellner \cite{Kellnermany} improved their construction to obtain a model with $\cfrak$-many pairwise different invariants. Later on, this technique was solidified in \cite{RosSh,KScreat,joyhalv} and, quite recently, A. Fischer, Goldstern, Kellner and Shelah \cite{FGKS} used it to construct a model where the cardinal invariants in Cicho\'n's diagram are separated into 5 different values. This model is quite special because it succeeds to separate 5 cardinals only on the right half of the diagram, while so far it has been possible to separate this half into three values with FS iteration techniques. On the other hand, these creature constructions are typically $\omega^\omega$-bounding, so they do not work to separate many values on the left side of the diagram, which can be done by FS iterations. For instance, Goldstern, Shelah and the author \cite{GMS} produced a model by a FS iteration where all the cardinals on the left side of the diagram are pairwise different (another example of 6 values). It is also possible, with FS iterations, to produce models where infinitely many cardinal characteristics are pairwise different, see e.g. \cite{KO}.

In the last year V. Fischer, Friedman, Montoya and the author \cite{FFMM} constructed the first example of a 3D array of ccc posets, called a \emph{3D-coherent system}, to force a model where the cardinals in Cicho\'n's diagram are separated into 7 different values. In addition, the techniques of \cite{VFJB11} to preserve a mad family can be applied in this context, so $\bfrak=\afrak$ can also be forced, even more, this equality can be forced in many of the precedent instances of models separating Cicho\'n's diagram with FS iterations.

In the same spirit as \cite{mejia2}, we use 3D-coherent systems as in \cite{FFMM} to construct models where several classical cardinal invariants (others than those in Cicho\'n's diagram) are pairwise different. These examples follow directly from the theory presented in \cite{FFMM,mejia2} and are just very simple modifications of the models constructed in \cite{FFMM}. 

This paper is structured as follows. Section \ref{SecPre} is devoted to the preliminaries, that is, the notion of coherent systems and the preservation theory are reviewed from \cite{FFMM} plus some additional examples taken from \cite{mejia2}. In Section \ref{SecAppl} the applications of the preceding theory are presented. 

\section{Preliminaries}\label{SecPre}

\subsection{Coherent systems of FS iterations}\label{SubsecCoherent}

For posets $\Por$ and $\Qor$ the relation $\Por\lessdot\Qor$ means that $\Por$ is a complete subposet of $\Qor$. If $M$ is a transitive model of ZFC and $\Por\in M$, $\Por\lessdot_M\Qor$ means that $\Por$ is a subposet of $\Qor$ and that every maximal antichain of $\Por$ in $M$ is a maximal antichain of $\Qor$.

\newcommand{\sbf}{\mathbf{s}}
\newcommand{\tbf}{\mathbf{t}}
\newcommand{\matit}{\mathbf{m}}

\begin{definition}[Coherent system of FS iterations {\cite[Def. 3.2]{FFMM}}]\label{DefCoherent}
   A \emph{coherent system (of FS iterations)} $\sbf$ is composed by the following objects:
   \begin{enumerate}[(I)]
     \item a partially ordered set $I^\sbf$ and an ordinal $\pi^\sbf$,
     \item a system of posets $\la\Por^\sbf_{i,\xi}:i\in I^\sbf,\xi\leq\pi^\sbf\ra$ such that
           \begin{enumerate}[(i)]
              \item $\Por^\sbf_{i,0}\lessdot\Por^\sbf_{j,0}$ whenever $i\leq j$ in $I^\sbf$, and
              \item $\Por^\sbf_{i,\eta}$ is the direct limit of $\la\Por^\sbf_{i,\xi}:\xi<\eta\ra$ for each limit $\eta\leq\pi^\sbf$,
           \end{enumerate}
     \item a sequence $\la\Qnm^\sbf_{i,\xi}:i\in I^\sbf,\xi<\pi^\sbf\ra$ where each $\Qnm^\sbf_{i,\xi}$ is a $\Por^\sbf_{i,\xi}$-name for a poset, $\Por^\sbf_{i,\xi+1}=\Por^\sbf_{i,\xi}\ast\Qnm^\sbf_{i,\xi}$ and $\Por^\sbf_{j,\xi}$ forces $\Qnm^\sbf_{i,\xi}\lessdot_{V^{\Por^\sbf_{i,\xi}}}\Qnm^\sbf_{j,\xi}$ whenever $i\leq j$ in $I^\sbf$ and $\Por^\sbf_{i,\xi}\lessdot\Por^\sbf_{j,\xi}$.
   \end{enumerate}
   Note that, for a fixed $i\in I^\sbf$, the posets $\la\Por^\sbf_{i,\xi}:\xi\leq\pi^\sbf\ra$ are generated by an FS iteration $\la\Por'_{i,\xi},\Qnm'_{i,\xi}:\xi<1+\pi^\sbf\ra$ where $\Qnm'_{i,0}=\Por^\sbf_{i,0}$ and $\Qnm'_{i,1+\xi}=\Qnm^\sbf_{i,\xi}$ for all $\xi<\pi^\sbf$. Therefore (by induction) $\Por'_{i,1+\xi}=\Por_{i,\xi}$ for all $\xi\leq\pi^\sbf$ and, thus, $\Por^\sbf_{i,\xi}\lessdot\Por^\sbf_{i,\eta}$ whenever $\xi\leq\eta\leq\pi^\sbf$.

   On the other hand, by Lemma  \ref{parallellimits}, $\Por^\sbf_{i,\xi}\lessdot\Por^\sbf_{j,\xi}$ whenever $i\leq j$ in $I^\sbf$ and $\xi\leq\pi^\sbf$.

   For $j\in I^\sbf$ and $\eta\leq\pi^\sbf$ we write $V^\sbf_{j,\eta}$ for the $\Por^\sbf_{j,\eta}$-generic extensions. Concretely, if $G$ is $\Por^\sbf_{j,\eta}$-generic over $V$, $V^\sbf_{j,\eta}:=V[G]$ and $V^\sbf_{i,\xi}:=V[\Por^\sbf_{i,\xi}\cap G]$ for all $i\leq j$ in $I^\sbf$ and $\xi\leq\eta$. Note that $V^\sbf_{i,\xi}\subseteq V^\sbf_{j,\eta}$.

   We say that the coherent system $\sbf$ has the \emph{ccc} if, additionally, $\Por^\sbf_{i,0}$ has the ccc and $\Por^\sbf_{i,\xi}$ forces that $\Qnm^\sbf_{i,\xi}$ has the ccc for each $i\in I^\sbf$ and $\xi<\pi^\sbf$. This implies that $\Por^\sbf_{i,\xi}$ has the ccc for all $i\in I^\sbf$ and $\xi\leq\pi^\sbf$.

   We consider the following particular cases.
   \begin{enumerate}[(1)]
      \item When $I^\sbf$ is a well-ordered set, we say that $\sbf$ is a \emph{2D-coherent system (of FS iterations)}.
      \item If $I^\sbf$ is of the form $\{i_0,i_1\}$ ordered as $i_0<i_1$, we say that $\sbf$ is a \emph{coherent pair (of FS iterations)}.
      \item If $I^\sbf=\gamma^\sbf\times\delta^\sbf$ where $\gamma^\sbf$ and $\delta^\sbf$ are ordinals and the order of $I^\sbf$ is defined as $(\alpha,\beta)\leq(\alpha',\beta')$ iff $\alpha\leq\alpha'$ and $\beta\leq\beta'$,
          we say that $\sbf$ is a \emph{3D-coherent system (of FS iterations).}
   \end{enumerate}

   For a coherent system $\sbf$ and a set $J\subseteq I^\sbf$, $\sbf|J$ denotes the coherent system with $I^{\sbf|J}=J$, $\pi^{\sbf|J}=\pi^\sbf$ and the posets and names corresponding to (II) and (III) defined as for $\sbf$. And if $\eta\leq\pi^\sbf$, $\sbf\frestr\eta$ denotes the coherent system with $I^{\sbf\upharpoonright\eta}=I^\sbf$, $\pi^{\sbf\upharpoonright\eta}=\eta$ and the posets for (II) and (III) defined as for $\sbf$. Note that, if $i_0<i_1$ in $I^\sbf$, then $\sbf|\{i_0,i_1\}$ is a coherent pair and $\sbf|\{i_0\}$ corresponds just to the FS iteration $\la\Por'_{i_0,\xi},\Qnm'_{i_0,\xi}:\xi<1+\pi^\sbf\ra$ (see the comment after (III)).

   If $\tbf$ is a 3D-coherent system, for $\alpha<\gamma^\tbf$, $\tbf_\alpha:=\tbf|\{(\alpha,\beta):\beta<\delta^\tbf\}$ is a 2D-coherent system where $I^{\tbf_\alpha}$ has order type $\delta^\tbf$. For $\beta<\delta^\tbf$, $\tbf^\beta:=\tbf|\{(\alpha,\beta):\alpha<\delta^\tbf\}$ is a 2D-coherent system where $I^{\tbf^\beta}$ has order type $\gamma^\tbf$.

   In particular, the upper indices $\sbf$ are omitted when there is
   no risk of ambiguity.
\end{definition}

\begin{lemma}\label{GenNoNewReals}
   Let $\matit$ be a ccc 2D-coherent system with $I^\matit=\gamma+1$ an ordinal and $\pi^\matit=\pi$. Assume that
   \begin{enumerate}[(i)]
     \item $\gamma$ has uncountable cofinality,
     \item $\Por_{\gamma,0}$ is the direct limit of $\la\Por_{\alpha,0}:\alpha<\gamma\ra$, and
     \item for any $\xi<\pi$, $\Por_{\gamma,\xi}$ forces ``$\Qnm_{\gamma,\xi}=\bigcup_{\alpha<\gamma}\Qnm_{\alpha,\xi}$" whenever $\Por_{\gamma,\xi}$ is the direct limit of $\la\Por_{\alpha,\xi}:\alpha<\gamma\ra$.
   \end{enumerate}
   Then, for any $\xi\leq\pi$, $\Por_{\gamma,\xi}$ is the direct limit of $\la\Por_{\alpha,\xi}:\alpha<\gamma\ra$. In particular, $\Por_{\gamma,\xi}$ forces that $\R\cap V_{\gamma,\xi}=\bigcup_{\alpha<\gamma}\R\cap V_{\alpha,\xi}$.
\end{lemma}

\subsection{Preservation theory}\label{SubsecPres}

\newcommand{\Rbf}{\mathbf{R}}

\begin{definition}\label{ContextFS}
   $\Rbf:=\la X,Y,\sqsubset\ra$ is a \emph{Polish relational system (Prs)} if the following is satisfied:
   \begin{enumerate}[(i)]
      \item $X$ is a perfect Polish space,
      \item $Y$ is a non-empty analytic subspace of some Polish space $Z$ and
      \item $\sqsubset=\bigcup_{n<\omega}\sqsubset_n$ for some increasing sequence $\la\sqsubset_n\ra_{n<\omega}$ of closed subsets of $X\times Z$ such that
            $(\sqsubset_n)^y=\{x\in X: x\sqsubset_n y\}$ is nwd (nowhere dense) for all $y\in Y$.
   \end{enumerate}
   For $x\in X$ and $y\in Y$, $x\sqsubset y$ is often read \emph{$y$ $\sqsubset$-dominates $x$}. A family $F\subseteq X$ is \emph{$\Rbf$-unbounded} if there is \underline{no} real in $Y$ that $\sqsubset$-dominates every member of $F$. Dually, $D\subseteq Y$ is a \emph{$\Rbf$-dominating} family if every member of $X$ is $\sqsubset$-dominated by some member of $D$. $\bfrak(\Rbf)$ denotes the least size of a $\Rbf$-unbounded family and $\dfrak(\Rbf)$ is the least size of a $\Rbf$-dominating family.

   Say that $x\in X$ is \emph{$\Rbf$-unbounded over a set $A$} if $x\not\sqsubset y$ for all $y\in Y\cap A$. Given a cardinal $\lambda$ say that $F\subseteq X$ is \emph{$\lambda$-$\Rbf$-unbounded} if, for any $A\subseteq Y$ of size $<\lambda$, there is an $x\in F$ which is $\Rbf$-unbounded over $A$. On the other hand, say that $F$ is \emph{strongly $\lambda$-$\Rbf$-unbounded} if $|F|\geq\lambda$ and $|\{x\in F:x\sqsubset y\}|<\lambda$ for all $y\in Y$.
\end{definition}

When $\lambda$ is regular, any strongly $\lambda$-$\Rbf$-unbounded family is $\lambda$-$\Rbf$-unbounded.

By (iii), $\la X,\Mwf(X),\in\ra$ is Tukey-Galois below $\Rbf$ where $\Mwf(X)$ denotes the $\sigma$-ideal of meager subsets of $X$. Therefore, $\bfrak(\Rbf)\leq\non(\Mwf)$ and $\cov(\Mwf)\leq\dfrak(\Rbf)$. Fix, for this subsection, a Prs $\Rbf=\la X,Y,\sqsubset\ra$.

\begin{definition}[Judah and Shelah {\cite{jushe}}]\label{DefGood}
   Let $\theta$ be a cardinal. A poset $\Por$ is \emph{$\theta$-$\Rbf$-good} if, for any $\Por$-name $\dot{h}$ for a real in $Y$, there is a non-empty $H\subseteq Y$ of size $<\theta$ such that $\Vdash x\not\sqsubset\dot{h}$ for any $x\in X$ that is $\Rbf$-unbounded over $H$.

   Say that $\Por$ is \emph{$\Rbf$-good} when it is $\aleph_1$-$\Rbf$-good.
\end{definition}

Definition \ref{DefGood} describes a property, respected by FS iterations, to preserve specific types of $\Rbf$-unbounded families. Concretely, when $\theta$ is an uncountable regular cardinal,
\begin{enumerate}[(a)]
  \item any $\theta$-cc $\theta$-$\Rbf$-good poset preserves every $\lambda$-$\Rbf$-unbounded family from the ground model when $\lambda\geq\theta$, it preserves strongly $\lambda$-$\Rbf$-unbounded families when $\cf(\lambda)\geq\theta$, and
  \item FS iterations of $\theta$-cc $\theta$-$\Rbf$-good posets produce $\theta$-$\Rbf$-good posets.
\end{enumerate}
Posets that are $\theta$-$\Rbf$-good work to preserve $\bfrak(\Rbf)$ small and $\dfrak(\Rbf)$ large since, whenever $F$ is a $\lambda$-$\Rbf$-unbounded family with $\lambda\geq2$, $\bfrak(\Rbf)\leq|F|$ and $\lambda\leq\dfrak(\Rbf)$.

Clearly, $\theta$-$\Rbf$-good implies $\theta'$-$\Rbf$-good whenever $\theta\leq\theta'$ and any poset completely embedded into a $\theta$-$\Rbf$-good poset is also $\theta$-$\Rbf$-good.

\begin{lemma}[{\cite[Lemma 4]{mejia}}]\label{smallGood}
   If $\theta$ is regular, any poset of size $<\theta$ is $\theta$-$\Rbf$-good. In particular, Cohen forcing is $\Rbf$-good.
\end{lemma}

\begin{lemma}\label{Cohen-unb}
  If $\theta$ is an uncountable regular cardinal, $\nu\geq\theta$ is a cardinal with $\cf(\nu)\geq\theta$ and $\mathds{P}_{\nu}=\langle\mathds{P}_{\alpha},\dot{\mathds{Q}}_{\alpha}\rangle_{\alpha<\nu}$ is a FS iteration where each $\dot{\mathds{Q}}_{\alpha}$ is forced (by $\Por_{\alpha}$) to be $\theta$-cc and non-trivial, then $\Por_{\nu}$ adds a strongly $\nu$-$\Rbf$-unbounded family of size $\nu$.
\end{lemma}
\begin{proof}
  The Cohen reals (in $X$) added at the limit steps of the iteration forms a strongly
  $\nu$-$\Rbf$-unbounded family of size $\nu$.
\end{proof}

\begin{theorem}\label{FSpres}
  Let $\theta$ be an uncountable regular cardinal, $\delta\geq\theta$ an ordinal, and let $\mathds{P}_{\delta}=\langle\mathds{P}_{\alpha},\dot{\mathds{Q}}_{\alpha}\rangle_{\alpha<\delta}$ be a FS iteration such that, for each $\alpha<\delta$, $\Qnm_{\alpha}$ is a $\Por_{\alpha}$-name of a  non-trivial $\theta$-$\Rbf$-good $\theta$-cc poset. Then:
\begin{enumerate}[(a)]
    \item For any cardinal $\nu\in[\theta,\delta]$ with $\cf(\nu)\geq\theta$, $\Por_\nu$ adds a strongly $\nu$-$\Rbf$-unbounded family of size $\nu$ which is still strongly $\nu$-$\Rbf$-unbounded in the $\Por_\delta$-extension.
    \item For any cardinal $\lambda\in[\theta,\delta]$, $\Por_\lambda$ adds a $\lambda$-$\Rbf$-unbounded family of size $\lambda$ which is still $\lambda$-$\theta$-unbounded in the $\Por_\delta$-extension.
    \item $\mathds{P}_{\delta}$ forces that $\bfrak(\Rbf)\leq\theta$ and  $|\delta|\leq\dfrak(\Rbf)$.
\end{enumerate}
\end{theorem}
\begin{proof}
   See e.g. \cite[Thm. 4.15]{CM}.
\end{proof}

Throughout this subsection, fix $M\subseteq N$ transitive models of ZFC and a Prs $\Rbf=\la X,Y,\sqsubset\ra$ coded in $M$. Recall that $\Sor$ is a \emph{Suslin ccc poset} if it is a
$\boldsymbol{\Sigma}^1_1$ subset of some
Polish space and both its order and
incompatibility relations are $\boldsymbol{\Sigma}^1_1$. Note that if
$\Sor$ is coded in $M$ then $\Sor^M\lessdot_M\Sor^N$.

\begin{lemma}[{\cite[Thm. 7]{mejia}}]\label{Suslingoodparallel}
   Let $\Sor$ be a Suslin ccc poset coded in $M$. If $M\models``\Sor$ is $\Rbf$-good" then, in $N$, $\Sor^N$ forces that every real in $X^N$ which is $\Rbf$-unbounded over $M$ is $\Rbf$-unbounded over $M^{\Sor^M}$.
\end{lemma}

Given a non-empty set $\Gamma$, denote the random algebra $\Bor_\Gamma:=\Bwf(2^{\Gamma\times\omega})/\Nwf(2^{\Gamma\times\omega})$ where $\Bwf(2^{\Gamma\times\omega})$ is the $\sigma$-algebra generated by sets of the form $[s]:=\{x\in2^{\Gamma\times\omega}:s\subseteq x\}$ for $s\in\Cor_\Gamma$, the class of finite partial functions from $\Gamma\times\omega$ into $2$, and $\Nwf(2^{\Gamma\times\omega})$ is the $\sigma$-ideal generated by the measure zero sets in $\Bwf(2^{\Gamma\times\omega})$ (with respect to the standard product measure). Note that $\Bor_\Gamma$ can be seen as the direct limit of posets of the form $\Bor_\Omega$ for $\Omega\subseteq\Gamma$ countable. Put $\Bor:=\Bor_\omega$ and $\Cor:=\Cor_\omega$.

\begin{corollary}\label{Presrandom}
   Let $\Gamma\in M$ be a non-empty set. If $M\models``\Bor_\Gamma$ is $\Rbf$-good" then $\Bor_\Gamma^N$, in $N$, forces that every real in $X^N$ which is $\Rbf$-unbounded over $M$ is $\Rbf$-unbounded over $M^{\Bor_\Gamma^M}$.
\end{corollary}

\begin{lemma}[{\cite[Lemma 11]{VFJB11}}, see also {\cite[Lemma 5.13]{mejia-temp}}]\label{Fixedparallel}
   Assume $\Por\in M$ is a poset. Then, in $N$, $\Por$ forces that every real in $X^N$ which is $\Rbf$-unbounded over $M$ is $\Rbf$-unbounded over $M^{\Por}$.
\end{lemma}

\begin{lemma}[Blass and Shelah {\cite{blassmatrix}}, {\cite[Lemmas 10, 12 and 13]{VFJB11}}]\label{parallellimits}
   Let $\sbf$ be a coherent pair of FS iterations as in Definition \ref{DefCoherent}(2). Then, $\Por_{i_0,\xi}\lessdot\Por_{i_1,\xi}$ for all $\xi\leq\pi$.

   Moreover, if $\dot{c}$ is a $\Por_{i_1,0}$-name of a real in $X$, $\pi$ is limit and $\Por_{i_1,\xi}$ forces that $\dot{c}$ is $\Rbf$-unbounded over $V_{i_0,\xi}$ for all $\xi<\pi$, then $\Por_{i_1,\pi}$ forces that $\dot{c}$ is $\Rbf$-unbounded over $V_{i_0,\pi}$.
\end{lemma}

\subsection{Examples of preservation properties}\label{SubsecExm}

\newcommand{\Ed}{\mathbf{Ed}}
\newcommand{\Dbf}{\mathbf{D}}
\newcommand{\Edb}{\mathbf{Ed}_b}
\newcommand{\Lc}{\mathbf{Lc}}
\newcommand{\Sp}{\mathbf{Sp}}
\newcommand{\FSp}{\mathbf{Fp}}
\newcommand{\Lor}{\mathds{L}}
\newcommand{\Id}{\mathbf{Id}}

\begin{example}\label{Ed}
   Consider the Prs $\Ed:=\la\omega^\omega,\omega^\omega,\neq^*\ra$ where $x\neq^*y$ iff $x$ and $y$ are eventually different, that is, $x(i)\neq y(i)$ for all but finitely many $i<\omega$.
   By \cite[Thm. 2.4.1 and 2.4.7]{judabarto}, $\bfrak(\Ed)=\non(\Mwf)$ and $\dfrak(\Ed)=\cov(\Mwf)$.
\end{example}

\begin{example}\label{Dom}
  Let $\Dbf:=\la\omega^\omega,\omega^\omega,\leq^*\ra$ be the Prs where $x\leq^*y$ iff $x(i)\leq y(i)$ for all but finitely many $i<\omega$. Clearly, $\bfrak(\Dbf)=\bfrak$ and $\dfrak(\Dbf)=\dfrak$.

  Miller \cite{Mi} proved that $\Eor$, the standard $\sigma$-centered poset that adds an eventually different real in $\omega^\omega$, is $\Dbf$-good. Furthermore, $\omega^\omega$-bounding posets, like the random algebra, are $\Dbf$-good.
\end{example}

\begin{example}\label{Edb}
   Let $b:\omega\to\omega\menos\{0\}$ such that $b\not\leq^*1$ and let $\Edb:=\la\prod b,\prod b,\neq^*\ra$ be the Prs where $\prod b:=\prod_{i<\omega}b(i)$. If $\sum_{i<\omega}\frac{1}{b(i)}<+\infty$ then $\Edb$ is Tukey-Galois below $\la\Nwf(\prod b),\prod b,\not\ni\ra$ (for $x\in\prod b$ the set $\{y\in\prod b:\neg(x\neq^* y)\}$ has measure zero with respect to the standard Lebesgue measure on $\prod b$), so $\cov(\Nwf)\leq\bfrak(\Edb)$ and $\dfrak(\Edb)\leq\non(\Nwf)$.

   By a similar argument as in \cite[Lemma $1^*$]{Br-Cichon}, any $\nu$-centered poset is $\theta$-$\Edb$-good for any regular $\theta$ and $\nu<\theta$ infinite. In particular, $\sigma$-centered posets are $\Edb$-good.
\end{example}

\begin{example}\label{Loc}
   For each $k<\omega$ let $\id^k:\omega\to\omega$ such that $\id^k(i)=i^k$ for all $i<\omega$ and put $\Hwf:=\{\id^{k+1}:k<\omega\}$. Let $\Lc:=\la\omega^\omega,\Swf(\omega,\Hwf),\in^*\ra$ be the Prs where
   \[\Swf(\omega,\Hwf):=\{\varphi:\omega\to[\omega]^{<\aleph_0}:\exists{h\in\Hwf}\forall{i<\omega}(|\varphi(i)|\leq h(i))\},\]
   and $x\in^*\varphi$ iff $\exists{n<\omega}\forall{i\geq n}(x(i)\in \varphi(i))$, which is read \emph{$x$ is localized by $\varphi$}. As a consequence of Bartoszy\'nski's characterization of measure (see \cite[Thm. 2.3.9]{judabarto}), $\bfrak(\Lc)=\add(\Nwf)$ and $\dfrak(\Lc)=\cof(\Nwf)$.

   Any $\nu$-centered poset is $\theta$-$\Lc$-good for any regular $\theta$ and $\nu<\theta$ infinite (see \cite{jushe}) so, in particular, $\sigma$-centered posets are $\Lc$-good. Moreover, subalgebras (not necessarily complete) of random forcing are $\Lc$-good as a consequence of a result of Kamburelis \cite{kamburelis}.
\end{example}

\begin{example}\label{Split}
   For $a,b\in[\omega]^{\aleph_0}$ define $a\varpropto b$ iff either $b\subseteq^* a$ or $b\subseteq^*\omega\menos a$. It is clear that $\Sp:=\la[\omega]^{\aleph_0},[\omega]^{\aleph_0},\varpropto\ra$ is a Prs. Note that $a\not\varpropto b$ iff \emph{$a$ splits $b$}, that is, $a\cap b$ and $b\menos a$ are infinite, so the \emph{splitting number} $\sfrak=\bfrak(\Sp)$ and the \emph{reaping number} $\rfrak=\dfrak(\Sp)$.

   Baumgartner and Dordal \cite{baudor} (see also \cite[Main Lemma 3.8]{brendlebog}) proved that Hechler forcing $\Dor$ (for adding a dominating real) is $\Sp$-good. On the other hand, any poset that adds random reals is \underline{not} $\Sp$-good (see \cite[Subsect. 2.2]{mejia2}).
\end{example}

\newcommand{\IP}{\mathrm{IP}}

\begin{example}[{\cite[Example 2.19]{mejia2}}]\label{FinSplit}
   Let $\IP$ be the class of interval partitions on $\omega$.
   For $a\in[\omega]^{\aleph_0}$ and $\bar{J}\in\IP$ define $a\rhd\bar{J}$ iff either $\forall^\infty{k<\omega}(I_k\nsubseteq a)$ or $\forall^\infty{k<\omega}(I_k\nsubseteq \omega\menos a))$. The relation $a\ntriangleright\bar{J}$ is often read \emph{$a$ splits $\bar{J}$}. Clearly, $\FSp:=\la[\omega]^{\aleph_0},\IP,\rhd\ra$ is a Prs, $\bfrak(\FSp)=\max\{\bfrak,\sfrak\}$ and $\dfrak(\FSp)=\min\{\dfrak,\rfrak\}$ (see \cite{KWsplit}). The author proved in \cite[Lemma 2.20]{mejia2} that every $\Dbf$-good poset is $\FSp$-good. In particular, the random algebras and $\Eor$ are $\FSp$-good.
\end{example}

Now, we review Brendle's  and Fischer's \cite{VFJB11} technique for preserving mad families. Fix transitive models $M\subseteq N$ of ZFC.

\begin{definition}[{\cite[Def. 2]{VFJB11}}]\label{DefpresDiag}
   Let $A=\la a_z:z\in\Omega\ra\in M$ be a family of infinite subsets of $\omega$ and $a^*\in[\omega]^{\aleph_0}$ (not necessarily in $M$). Say that \emph{$a^*$ diagonalizes $M$ outside $A$} if, for all $h\in M$, $h:\omega\times[\Omega]^{<\aleph_0}\to\omega$ and for any $m<\omega$, there are $i\geq m$ and $F\in[\Omega]^{<\aleph_0}$ such that $[i,h(i,F))\menos\bigcup_{z\in F}a_z\subseteq a^*$.
\end{definition}

For $A\subseteq\Pwf(\omega)$, denote by $\Iwf(A)$ the ideal generated by $A\cup\mathrm{Fin}$.

\begin{lemma}[{\cite[Lemma 3]{VFJB11}}]\label{DiagMain}
   If $A\in M$ and $a^*$ diagonalizes $M$ outside $A$ then $|a^*\cap x|=\aleph_0$ for any $x\in M\menos\Iwf(A)$.
\end{lemma}

\begin{lemma}[{\cite[Lemma 12]{VFJB11}}]\label{limitMAD}
   Let $\sbf$ be a coherent pair of FS iterations, $\dot{A}$ a $\Por_{i_0,0}$-name of a family of infinite subsets of $\omega$ and $\dot{a}^*$ a $\Por_{i_1,0}$-name for an infinite subset of $\omega$ such that
   \[\Vdash_{\Por_{i_1,\xi}}\textrm{``$\dot{a}^*$ diagonalizes $V_{i_0,\xi}$ outside $\dot{A}$"}\]
   for all $\xi<\pi$. Then, $\Por_{i_0,\pi}\lessdot\Por_{i_1,\pi}$ and $\Vdash_{\Por_{i_1,\pi}}$ ``$\dot{a}^*$ diagonalizes $V_{i_0,\pi}$ outside $\dot{A}$".
\end{lemma}

\begin{lemma}\label{Diag-onestep}
  Let $\Por\in M$ and $\Qor\in N$ be posets, $A\in M$. For each of the following cases, $N$ satisfies that $\Qor$ forces that any real in $[\omega]^{\aleph_0}\cap N$ that diagonalizes $M$ outside $A$ also diagonalizes $M^\Por$ outside $A$.
  \begin{enumerate}[(i)]
     \item ({\cite[Lemma 11]{VFJB11}}) $\Por=\Qor$.
     \item ({\cite[Lemmas 4.8, 4.10, Cor. 4.11]{FFMM}}) $\Por=\Sor^M$ and $\Qor=\Sor^N$ if $\Sor\in\{\Eor\}\cup\{\Bor_\Gamma:\Gamma\in M\}$.
  \end{enumerate}
\end{lemma}

For a filter base $F$ on $\omega$ that contains $\{\omega\menos n:n<\omega\}$, denote by $\Lor_F$ \emph{Laver forcing with $F$} and by $\Mor_F$ \emph{Mathias forcing with $F$}.

\begin{lemma}[{\cite[Crucial Lemma 7]{VFJB11}}]\label{mathias}
   In $M$, let $U$ be a non-principal ultrafilter on $\omega$ and $A\subseteq[\omega]^{\aleph_0}$. If $a^*\in[\omega]^{\aleph_0}\cap N$ diagonalizes $M$ outside $A$ then, in $N$, there is an ultrafilter $U'$ containing $U$ such that $\Mor_U\lessdot_M\Mor_{U'}$ and $\Mor_{U'}$ forces that $a^*$ diagonalizes $M^{\Mor_U}$ outside $A$.
\end{lemma}

\newcommand{\Hor}{\mathds{H}}

Hechler \cite{Hechlermad} defined an $\omega_1$-precaliber poset $\Hor_\Gamma$ for a set $\Gamma$ that adds an a.d. family $A^*(\Gamma)=\{a^*_i:i\in\Gamma\}$. When $\Gamma$ is uncountable, $A^*(\Gamma)$ is forced to be a mad family. If $\Omega\subseteq\Omega'$ then $\Hor_\Omega\lessdot\Hor_{\Omega'}$.

\begin{lemma}[{\cite[Lemma 4]{VFJB11}}]\label{HechlerMad}
   Let $\Omega$ be a set, $z^*\in\Omega$ and $A:=\{a_z:z\in\Omega\}$ the a.d. family added by $\Hor_\Omega$. Then, $\Hor_\Omega$ forces that $a_{z^*}$ diagonalizes $V^{\Hor_{\Omega\menos\{z^*\}}}$ outside $A\frestr(\Omega\menos\{z^*\})$
\end{lemma}

We finish this section with a result related to the groupwise-density number $\gfrak$.

\begin{lemma}[Blass {\cite[Thm. 2]{blass-super}}, see also {\cite[Lemma 1.17]{BrHejnice}}]\label{lemmagfrak}
   Let $\theta$ be an uncountable regular cardinal and $\langle V_\alpha\rangle_{\alpha\leq\theta}$ an increasing sequence of transitive models of ZFC such that
   \begin{enumerate}[(i)]
      \item $[\omega]^{\aleph_0}\cap(V_{\alpha+1}\menos V_{\alpha})\neq\varnothing$,
      \item $\langle[\omega]^{\aleph_0}\cap V_\alpha\rangle_{\alpha<\theta}\in V_\theta$ and
      \item $[\omega]^{\aleph_0}\cap V_\theta=\bigcup_{\alpha<\theta}[\omega]^{\aleph_0}\cap V_\alpha$.
   \end{enumerate}
   Then, in $V_\theta$, $\gfrak\leq\theta$.
\end{lemma}

\section{Applications}\label{SecAppl}

For ordinals $\gamma$ and $\delta$, fix the following ccc 2D-coherent systems.
\begin{enumerate}[(1)]

\item The system $\matit(\gamma,\delta)$ where

\begin{enumerate}[(i)]
   \item $I^{\matit(\gamma,\delta)}=\gamma+1$, $\pi^{\matit(\gamma,\delta)}=\delta$,
   \item $\Por^{\matit(\gamma,\delta)}_{\alpha,0}=\Cor_\alpha$ for each $\alpha\leq\gamma$, and
   \item $\Qnm^{\matit(\gamma,\delta)}_{\alpha,\beta}=\dot{\Cor}$ for all $\beta<\delta$.
\end{enumerate}
Note that $\Por^{\matit(\gamma,\delta)}_{\alpha,\beta}\simeq\Cor_{\alpha\times\beta}$ for all $\alpha\leq\gamma$ and $\beta\leq\delta$. For a perfect Polish space $X$, fix $\dot{c}_{X,\alpha}$ a $\Por^{\matit(\gamma,\delta)}_{\alpha+1,0}$-name of a Cohen real in $X$ over $V_{\alpha,0}$, and $\dot{c}_X^\beta$ a $\Por^{\matit(\gamma,\delta)}_{0,\beta+1}$-name of a Cohen real in $X$ over $V_{0,\beta}$ (recall that, for each $\alpha\leq\gamma$, $\Por^{\matit(\gamma,\delta)}_{\alpha,\beta+1}$ forces that $\dot{c}^\beta_X$ is Cohen over $V_{\alpha,\beta}$). Thus, for any Prs $\Rbf=\la X,Y,\sqsubset\ra$, $\Por^{\matit(\gamma,\delta)}_{\alpha+1,0}$ forces that $\dot{c}_{X,\alpha}$ is $\Rbf$-unbounded over $V_{\alpha,0}$ and, by the results of Subsection \ref{SubsecPres}, $\Por^{\matit(\gamma,\delta)}_{\alpha+1,\delta}$ forces that $\dot{c}_{X,\alpha}$ is $\Rbf$-unbounded over $V_{\alpha,\delta}$. Therefore, if $\gamma$ is a cardinal of uncountable cofinality then $\Por^{\matit(\gamma,\delta)}_{\gamma,\delta}$ forces that $\{\dot{c}_{X,\alpha}:\alpha<\gamma\}$ is strongly $\gamma$-$\Rbf$-unbounded.

\item The system $\matit^*(\gamma,\delta)$ where
\begin{enumerate}[(i)]
   \item $I^{\matit^*(\gamma,\delta)}=\gamma+1$, $\pi^{\matit(\gamma,\delta)}=\delta$,
   \item $\Por^{\matit^*(\gamma,\delta)}_{\alpha,0}=\Hor_\alpha$ for each $\alpha\leq\gamma$, and
   \item $\Qnm^{\matit(\gamma,\delta)}_{\alpha,\beta}=\dot{\Cor}$ for all $\beta<\delta$.
\end{enumerate}
Note that $\Por^{\matit^*(\gamma,\delta)}_{\alpha,\beta}\simeq\Hor_\alpha\times\Cor_{\beta}$ for all $\alpha\leq\gamma$ and $\beta\leq\delta$. Let $\dot{A}^*(\gamma)=\{\dot{a}^*_\alpha:\alpha<\gamma\}$ be the generic a.d. family added by $\Hor_\gamma$. For any $\alpha<\gamma$, $\Por^{\matit^*(\gamma,\delta)}_{\alpha+1,0}$ forces that $\dot{a}^*_\alpha$ diagonalizes $V_{\alpha,0}$ outside $\dot{A}^*(\alpha):=\{\dot{a}^*_\iota:\iota<\alpha\}$ and, by the results of the last part of Subsection \ref{SubsecExm}, $\Por^{\matit^*(\gamma,\delta)}_{\alpha+1,\delta}$ forces that $\dot{a}^*_\alpha$ diagonalizes $V_{\alpha,\delta}$ outside $\dot{A}^*(\alpha)$. Therefore, when $\gamma$ is an ordinal of uncountable cofinality, $\Por^{\matit^*(\gamma,\delta)}_{\alpha,\delta}$ forces that $\dot{A}^*(\gamma)$ is a mad family. Fix $\dot{c}^\beta_X$ as in (1).
\end{enumerate}

Fix uncountable regular cardinals $\theta_0\leq\theta_1\leq\kappa\leq\mu\leq\nu$, a cardinal $\lambda\geq\nu$ such that $\lambda^{<\kappa}=\lambda$, and a bijection $g=(g_0,g_1,g_2):\lambda\to\kappa\times\nu\times\lambda$. Denote $\pi=\lambda\nu\mu$ (ordinal product).

\newcommand{\MA}{\mathrm{MA}}

\begin{theorem}\label{coflarge}
   For any of the statements below there is a ccc poset forcing it.
   \begin{enumerate}[(a)]
      \item $\MA_{<\theta_0}$, $\add(\Nwf)=\theta_0$, $\pfrak=\afrak=\sfrak=\gfrak=\kappa$, $\cov(\Iwf)=\non(\Iwf)=\mu$ for $\Iwf\in\{\Mwf,\Nwf\}$, $\dfrak=\rfrak=\nu$, and $\cof(\Nwf)=\cfrak=\lambda$.
      \item $\MA_{<\theta_0}$, $\add(\Nwf)=\theta_0$, $\pfrak=\afrak=\cov(\Nwf)=\non(\Mwf)=\gfrak=\kappa$, $\cov(\Mwf)=\dfrak=\rfrak=\non(\Nwf)=\nu$, and $\cof(\Nwf)=\cfrak=\lambda$.
      \item $\MA_{<\theta_0}$, $\add(\Nwf)=\theta_0$, $\pfrak=\afrak=\sfrak=\cov(\Nwf)=\gfrak=\kappa$, $\cov(\Mwf)=\non(\Mwf)=\mu$, $\dfrak=\rfrak=\non(\Nwf)=\nu$, and $\cof(\Nwf)=\cfrak=\lambda$.
      \item $\MA_{<\theta_0}$, $\add(\Nwf)=\theta_0$, $\pfrak=\sfrak=\gfrak=\cov(\Nwf)=\kappa$, $\add(\Mwf)=\cof(\Mwf)=\afrak=\mu$, $\non(\Nwf)=\rfrak=\nu$, and $\cof(\Nwf)=\cfrak=\lambda$.
      \item $\MA_{<\theta_0}$, $\add(\Nwf)=\theta_0$, $\pfrak=\afrak=\cov(\Nwf)=\gfrak=\kappa$, $\sfrak=\non(\Mwf)=\cov(\Mwf)=\dfrak=\ufrak=\non(\Nwf)=\mu$ and $\cof(\Nwf)=\cfrak=\lambda$.
   \end{enumerate}
\end{theorem}
\begin{proof}
  \begin{enumerate}[(a)]
     \item Fix $\Delta=(\Delta_0,\Delta_1):\nu\mu\to\kappa\times\nu$ such that
           \begin{enumerate}[(i)]
              \item for any $\zeta<\nu\mu$, $\Delta_0(\xi)$ and $\Delta_1(\xi)$ are successor ordinals,
              \item for each $(\alpha,\beta)\in\kappa\times\nu$ the set $\{\zeta<\nu\mu:\Delta(\zeta)=(\alpha+1,\beta+1)\}$ is cofinal in $\nu\mu$.
           \end{enumerate}
     We construct a 3D-coherent system $\tbf$ with $I^\tbf=(\kappa+1)\times(\nu+1)$ and $\pi^\tbf=\pi=\lambda\nu\mu$ (ordinal product) by recursion on $\xi<\pi$. $\Por_{\alpha,\beta,0}=\Por^{\matit^*(\kappa,\nu)}_{\alpha,\beta}$ for all $\alpha\leq\kappa$ and $\beta\leq\nu$. Assume we have constructed the 3D system up to $\lambda\zeta$ for $\zeta<\nu\mu$. Enumerate $\{\Qnm^\zeta_{\alpha,\beta,\rho}:\rho<\lambda\}$ all the (nice) $\Por_{\alpha,\beta,\lambda\zeta}$-names for ccc posets of size $<\theta_0$ with underlying set a subset of $\theta_0$ and $\{\Rnm^\zeta_{\alpha,\beta,\rho}:\rho<\lambda\}$ all the (nice) $\Por_{\alpha,\beta,\lambda\zeta}$-names for $\sigma$-centered posets of size $<\kappa$ with underlying set a subset of $\kappa$. As in limit steps we just take the direct limit, we only need to take care of the successor stages for $\xi\in[\lambda\zeta,\lambda(\zeta+1))$.
     \begin{enumerate}[(1)]
       \item if $\xi=\lambda\zeta$ fix $\dot{U}_\zeta$ is a $\Por_{\Delta(\zeta),\xi}$-name for a no-principal ultrafilter on $\omega$ and put $\Qnm_{\alpha,\beta,\xi}=\Lor_{\dot{U}_\zeta}$ when $\alpha\in[\Delta_0(\zeta),\kappa]$ and $\beta\in[\Delta_1(\zeta),\nu]$, otherwise $\Qnm_{\alpha,\beta,\xi}$ is the trivial poset;
       \item if $\xi=\lambda\zeta+1$, $\Qnm_{\alpha,\beta,\xi}$ is a $\Por_{\alpha,\beta,\xi}$-name for $\Bor^{V_{\alpha,\beta,\xi}}$;
       \item if $\xi=\lambda\zeta+2+2\rho$ for some $\rho<\lambda$, $\Qnm_{\alpha,\beta,\xi}=\Qnm^\zeta_{g(\rho)}$ when $\alpha\in(g_0(\rho),\kappa]$, $\beta\in(g_1(\rho),\nu]$ and $\Por_{\kappa,\nu,\xi}$ forces $\Qnm^\zeta_{g(\rho)}$ to be ccc, otherwise $\Qnm_{\alpha,\beta,\xi}$ is the trivial poset;
       \item if $\xi=\lambda\zeta+2+2\rho+1$ for some $\rho<\lambda$, $\Qnm_{\alpha,\beta,\xi}=\Rnm^\zeta_{g(\rho)}$ when $\alpha\in(g_0(\rho),\kappa]$, $\beta\in(g_1(\rho),\nu]$ and $\Por_{\kappa,\nu,\xi}$ forces $\Rnm^\zeta_{g(\rho)}$ to be $\sigma$-centered, otherwise $\Qnm_{\alpha,\beta,\xi}$ is the trivial poset.
     \end{enumerate}
     Indeed, $\Por_{\alpha,\beta,\xi+1}=\Por_{\alpha,\beta,\xi}\ast\Qnm_{\alpha,\beta,\xi}$. To fix some terminology, we say that in a stage $\xi$ as in (1) we \emph{add a restricted Laver (with ultrafilter) generic real}. For (2) we say that we \emph{add a full random generic}, for (3) we say that we \emph{perform a counting argument of ccc posets of size $<\theta_0$}, and for (4) we say that we \emph{perform a counting argument of $\sigma$-centered posets of size $<\kappa$}.\footnote{This terminology will be used in further proofs to describe briefly similar forcing constructions.} Those counting arguments yield $\MA_{<\theta_0}$ and $\pfrak\geq\kappa$ in the final extension. In particular, $\Por_{\kappa,\nu,\pi}$ forces $\add(\Nwf)\geq\theta_0$.

     Note that this construction produces a FS iteration of length $\pi$ of $\theta_0$-$\Lc$-good posets. Therefore, by Theorem \ref{FSpres}, $\Por_{\kappa,\nu,\pi}$ forces $\add(\Nwf)\leq\theta_0$ and $\cof(\Nwf)=\cfrak=\lambda$.

     Recall that the base of this construction is the 2D-coherent system $m^*(\kappa,\nu)$. Note that $\{\dot{c}_{[\omega]^{\aleph_0},\alpha}:\alpha<\kappa\}$ is preserved to be a strongly $\kappa$-$\FSp$-unbounded family in $V_{\kappa,\nu,\pi}$, so $\sfrak\leq\bfrak(\FSp)\leq\kappa$. Also, $A^*(\kappa)$ is preserved to be a mad family in $V_{\kappa,\nu,\pi}$, so $\afrak\leq\kappa$. On the other hand, $\{\dot{c}^\beta_{[\omega]^{\aleph_0}}:\beta<\nu\}$ is preserved to be strongly $\nu$-$\FSp$-unbounded in $V_{\kappa,\nu,\pi}$, so $\min\{\dfrak,\rfrak\}=\dfrak(\FSp)\geq\nu$. Besides, $\gfrak\leq\kappa$ holds in $V_{\kappa,\nu,\pi}$ by Lemma \ref{lemmagfrak}.

     For each $\zeta<\nu\mu$ let $d_\zeta$ be the dominating real added by $\Lor_{U_\zeta}$ and $l_\zeta$ be the unsplitting real added by it. It is clear that $\{d_\zeta:\zeta<\nu\mu\}$ is a dominating family and $\{l_\zeta:\zeta<\nu\mu\}$ is a reaping family, so both $\dfrak$ and $\rfrak$ are below $\nu$.

     Finally, note that $\cov(\Iwf)=\non(\Iwf)=\mu$ for $\Iwf\in\{\Mwf,\Nwf\}$ holds in $V_{\kappa,\nu,\pi}$ because of the cofinally $\mu$-many Cohen and random reals added along the iteration.

     \item The construction can be performed in two ways. The first one is a 3D-coherent system $\tbf$ constructed with base $\matit^*(\kappa,\nu)$ and $\pi^\tbf=\lambda\nu$ where, at stages between $\lambda\zeta$ and $\lambda(\zeta+1)$ for $\zeta<\nu$, we add a restricted random generic, a restricted Laver (with ultrafilter) generic and use counting arguments of ccc posets of size $<\theta_0$ and $\sigma$-centered posets of size $<\kappa$. Here the function $\Delta$ (that indicates where the restricted generic reals are added) is defined as in (a) but with domain $\nu$.

         The second construction is a 2D-coherent system $\matit$ with $I^\matit=\nu+1$, $\pi^\matit=\lambda\nu\kappa$ where $\Por_{\alpha,0}=\Cor_\alpha$ for $\alpha\leq\nu$ and, at stages between $\lambda\zeta$ and $\lambda(\zeta+1)$ for $\zeta<\nu\kappa$, we add restricted generic reals and perform counting arguments as in $\tbf$ but, in addition, we use the steps of the form $\lambda\nu\eta$ to add a mad family of size $\kappa$. Assume that we have added a $\Por_{0,\lambda\nu\eta}$-name $\dot{A}(\eta)=\{\dot{a}_{\varepsilon}:\varepsilon<\eta\}$ for an a.d. family. Let $\dot{F}_\eta$ be a $\Por_{0,\lambda\nu\eta}$-name for the filter base generated by the complements of the members of $\dot{A}(\eta)$ and the co-finite subsets of $\omega$ and let $\Qnm_{\alpha,\lambda\nu\eta}=\Mor_{\dot{F}_\eta}$ for all $\alpha\leq\nu$. They add an infinite subset $\dot{a}_\eta\in V_{0,\lambda\nu\kappa+1}$ of $\omega$ (which does not depend on $\alpha$) that is a.d. with $\dot{A}(\eta)$ and that diagonalizes $V_{\nu,\lambda\nu\kappa}$ outside $\dot{A}(\eta)$. Therefore, $\{\dot{a}_\eta:\eta<\kappa\}$ is forced to be a mad family in the final model.

     \item Construct a 3D-coherent system based in $\matit^*(\kappa,\nu)$ with $\pi^\tbf=\pi$ where, at stages between $\lambda\zeta$ and $\lambda(\zeta+1)$, a full $\Eor$-generic real, restricted random and Laver (with ultrafilter) generic reals are added and counting arguments of ccc posets of size $<\theta_0$ and $\sigma$-centered posets of size $<\kappa$ are performed.

     \item Construct a 3D-coherent system based in $\matit(\kappa,\nu)$ with $\pi^\tbf=\pi$ where, at stages between $\lambda\zeta$ and $\lambda(\zeta+1)$, a full Hechler generic real, restricted random and Laver (with ultrafilter) generic reals are added and counting arguments of ccc posets of size $<\theta_0$ and $\sigma$-centered posets of size $<\kappa$ are performed. In addition, the stages of the form $\lambda\nu\eta$ for $\eta<\mu$ are used to add a mad family of size $\mu$ with an argument as in (b).

     \item Construct a 2D-coherent system $\matit$ with $I^\matit=\kappa+1$, $\pi^\matit=\lambda\mu$, $\Por_{\alpha,0}=\Hor_\alpha$ and, at stages between $\lambda\zeta$ and $\lambda\zeta+1$, add restricted Hechler and random generic reals and perform counting arguments of ccc posets of size $<\theta_0$ and $\sigma$-centered posets of size $<\kappa$. In addition, we do the following in stages of the form $\lambda\zeta$ for $\zeta<\mu$. Assume we have constructed $\Por_{0,\lambda\zeta}$-names $\{\dot{m}_{\varepsilon}:\varepsilon<\zeta\}$ for a $\subseteq^*$-decreasing sequence of infinite subsets of $\omega$. As in \cite{VFJB11}, by recursion on $\alpha\leq\kappa$, we use Lemma \ref{mathias} to construct a $\Por_{\alpha,\lambda\zeta}$-name $\dot{U}_{\alpha,\zeta}$ of an ultrafilter that contains $\{\dot{m}_{\varepsilon}:\varepsilon<\zeta\}$ and $\dot{U}_{\iota,\zeta}$ for all $\iota<\alpha$. Put $\Qnm_{\alpha,\lambda\zeta}=\Mor_{\dot{U}_{\alpha,\zeta}}$ and denote by $\dot{m}_{\zeta}$ the Mathias generic real it adds. In the final model, $\{\dot{m}_\zeta:\zeta<\mu\}$ is a base of an ultrafilter.
  \end{enumerate}
\end{proof}

\begin{theorem}\label{clarge}
   For any of the statements below there is a ccc poset forcing it.
   \begin{enumerate}[(a)]
      \item $\MA_{<\kappa}$, $\add(\Nwf)=\pfrak=\sfrak=\gfrak=\cov(\Nwf)=\kappa$, $\add(\Mwf)=\cof(\Mwf)=\afrak=\mu$, $\non(\Nwf)=\cof(\Nwf)=\rfrak=\nu$ and $\cfrak=\lambda$.
      \item $\MA_{<\kappa}$, $\add(\Nwf)=\pfrak=\sfrak=\gfrak=\afrak=\kappa$, $\cov(\Iwf)=\non(\Iwf)=\mu$ for $\Iwf\in\{\Mwf,\Nwf\}$, $\dfrak=\cof(\Nwf)=\rfrak=\nu$ and and $\cfrak=\lambda$.
      \item $\MA_{<\kappa}$, $\add(\Nwf)=\pfrak=\gfrak=\kappa$, $\cov(\Nwf)=\add(\Mwf)=\cof(\Mwf)=\afrak=\non(\Nwf)=\mu$, $\cof(\Nwf)=\nu$ and $\cfrak=\lambda$.
      \item $\MA_{<\kappa}$, $\add(\Nwf)=\pfrak=\afrak=\gfrak=\non(\Mwf)=\kappa$, $\cov(\Mwf)=\cof(\Nwf)=\rfrak=\nu$ and $\cfrak=\lambda$.
      \item $\MA_{<\kappa}$, $\add(\Nwf)=\pfrak=\gfrak=\cov(\Nwf)=\afrak=\kappa$, $\sfrak=\non(\Mwf)=\cov(\Mwf)=\cof(\Nwf)=\ufrak=\mu$ and $\cfrak=\lambda$.
   \end{enumerate}
\end{theorem}
\begin{proof}
   \begin{enumerate}[(a)]
      \item Construct a 3D-coherent system $\tbf$ with base $\matit(\kappa,\nu)$, $\pi^\tbf=\lambda\nu\mu$ such that, at stages between $\lambda\zeta$ and $\lambda(\zeta+1)$, a full Hechler real and restricted amoeba generic and Laver (with ultrafilter) generic reals are added and a counting argument of ccc posets of size $<\kappa$ is included. In addition, stages of the form $\lambda\nu\eta$ are used to add a mad family of size $\mu$.
      \item Construct a 3D-coherent system $\tbf$ with base $\matit^*(\kappa,\nu)$, $\pi^\tbf=\lambda\nu\mu$ such that, at stages between $\lambda\zeta$ and $\lambda(\zeta+1)$, a full random real, a restricted Laver (with ultrafilter) generic and a restricted amoeba generic are added and a counting argument of ccc posets of size $<\kappa$ is included.
      \item Construct a 3D-coherent system $\tbf$ with base $\matit(\kappa,\nu)$, $\pi^\tbf=\lambda\nu\mu$ such that, at stages between $\lambda\zeta$ and $\lambda(\zeta+1)$, a full random real, a full Hechler real and a restricted amoeba generic are added, and a counting argument of ccc posets of size $<\kappa$ is included. In addition, stages of the form $\lambda\nu\eta$ are used to add a mad family of size $\mu$.
      \item Construct a 3D-coherent system $\tbf$ with base $\matit^*(\kappa,\nu)$, $\pi^\tbf=\lambda\nu$ such that, at stages between $\lambda\zeta$ and $\lambda(\zeta+1)$, a restricted amoeba generic and a restricted Laver generic real are added and a counting argument of ccc posets of size $<\kappa$ is included. Like in Theorem \ref{coflarge}(b), the same model can be obtained by a 2D-coherent system.
      \item Construct a 2D-coherent system $\matit$ with $I^\matit=\kappa+1$, $\pi^\matit=\lambda\mu$, $\Por_{\alpha,0}=\Hor_\alpha$ and, at stages between $\lambda\zeta$ and $\lambda\zeta+1$, a restricted amoeba generic is added, a counting argument of ccc posets of size $<\kappa$ is included, and stages of the form $\lambda\zeta$ are used to add an ultrafilter base of size $\mu$ as in Theorem \ref{coflarge}(e).
   \end{enumerate}
\end{proof}

\begin{theorem}\label{nonlarge}
   For any of the statements below there is a ccc poset forcing it.
   \begin{enumerate}[(a)]
      \item $\MA_{<\theta_0}$, $\add(\Nwf)=\theta_0$, $\cov(\Nwf)=\theta_1$, $\pfrak=\sfrak=\gfrak=\afrak=\kappa$, $\non(\Mwf)=\cov(\Mwf)=\mu$, $\dfrak=\rfrak=\nu$ and $\non(\Nwf)=\cfrak=\lambda$.
      \item $\MA_{<\theta_0}$, $\add(\Nwf)=\theta_0$, $\cov(\Nwf)=\theta_1$, $\pfrak=\gfrak=\afrak=\non(\Mwf)=\kappa$, $\cov(\Mwf)=\dfrak=\rfrak=\nu$ and $\non(\Nwf)=\cfrak=\lambda$.
      \item $\MA_{<\theta_0}$, $\add(\Nwf)=\theta_0$, $\cov(\Nwf)=\theta_1$, $\pfrak=\gfrak=\afrak=\kappa$, $\sfrak=\non(\Mwf)=\cov(\Mwf)=\dfrak=\ufrak=\mu$ and $\non(\Nwf)=\cfrak=\lambda$.
   \end{enumerate}
\end{theorem}
\begin{proof}
  \begin{enumerate}[(a)]
     \item Construct a 3D-coherent system $\tbf$ with base $\matit^*(\kappa,\nu)$, $\pi^\tbf=\lambda\nu\mu$ such that, at stages between $\lambda\zeta$ and $\lambda(\zeta+1)$, a full $\Eor$-generic and a restricted Laver (with ultrafilter) generic real are added, and counting arguments of ccc posets of size $<\theta_0$, subalgebras of random forcing of size $<\theta_1$ and $\sigma$-centered posets of size $<\kappa$ are included.
     \item Construct a 3D-coherent system $\tbf$ with base $\matit^*(\kappa,\nu)$, $\pi^\tbf=\lambda\nu$ such that, at stages between $\lambda\zeta$ and $\lambda(\zeta+1)$, a restricted Laver (with ultrafilter) generic real is added, and counting arguments of ccc posets of size $<\theta_0$, subalgebras of random forcing of size $<\theta_1$ and $\sigma$-centered posets of size $<\kappa$ are included. As in Theorem \ref{coflarge}(b), the model can also be obtained by a 2D-coherent system.
     \item Construct a 2D-coherent system $\matit$ with $I^\matit=\kappa+1$, $\pi^\matit=\lambda\mu$, $\Por_{\alpha,0}=\Hor_\alpha$ and, at stages between $\lambda\zeta$ and $\lambda\zeta+1$, a restricted Hechler real is added, counting arguments of ccc posets of size $<\theta_0$, subalgebras of random forcing of size $<\theta_1$ and $\sigma$-centered posets of size $<\kappa$ are included, and stages of the form $\lambda\zeta$ are used to add an ultrafilter base of size $\mu$.
  \end{enumerate}
\end{proof}


\subsection*{Acknowledgements}
This paper was developed in the framework of the RIMS workshop 2016 on Infinite Combinatorics and Forcing Theory at Kyoto University, Japan. The author is very grateful with professor Teruyuki Yorioka for organizing such a wonderful workshop, for his support and for the chance to participate in and help during the workshop.

\bibliography{appl}
\bibliographystyle{alpha}


\end{document}